\documentclass{amsart}
\pdfoutput=1
\usepackage{amsmath,amsthm,amsfonts,latexsym,epsfig,longtable}

\newtheorem{thm}{Theorem}[section]
\newtheorem{lem}[thm]{Lemma}
\newtheorem{pro}[thm]{Proposition}

\theoremstyle{remark}
\newtheorem*{rem}{Remark}

\newcommand{\al}{\alpha}
\newcommand{\be}{\beta}

\newcommand{\Ga}{\Gamma}

\newcommand{\om}{\omega}

\newcommand{\Ta}{\Theta}

\newcommand{\cD}{\mathfrak{D}}

\newcommand{\cL}{\mathcal{L}}
\newcommand{\cE}{\mathfrak{E}}
\newcommand{\cf}{\mathfrak{f}}

\newcommand{\fT}{\mathfrak{T}}
\newcommand{\fR}{\mathfrak{R}}

\newcommand{\dmrjdel}[1]{}

\newcommand{\hook}{\mathfrak{h}}

\newcommand{\m}{\mathbf{m}}
\newcommand{\M}{\mathbf{M}}

\title{Hook Weighted Increasing Trees, Cayley Trees and Abel-Hurwitz Identities}

\date{\today}

\author{S.R. Carrell}
\address{Department of Combinatorics \& Optimization. University of Waterloo, Canada}
\email{srcarrel@uwaterloo.ca}

\begin{document}

\begin{abstract}

Recently F\'eray, Goulden and Lascoux gave a proof of a new hook summation formula for unordered increasing trees by means of a generalization of the
Pr\"ufer code for labelled trees and posed the problem of finding a bijection between weighted increasing trees and Cayley trees.
We give such a bijection, providing an answer to the problem posed by F\'eray, Goulden and Lascoux as well as showing a combinatorial connection to
the theory of tree volumes defined by Kelmans. In addition we give two simple proofs of the hook summation formula. As an application we describe how the hook
summation formula gives a combinatorial proof of a generalization of Abel and Hurwitz' theorem, originally proven by Strehl.

\end{abstract}

\maketitle

\section{Introduction}

We begin by fixing some terminology.
A tree $T$ is an acyclic connected graph and we denote by $V(T)$ the set of vertices of $T$ and $E(T)$ the set of edges. A tree is said to
be rooted if one of its vertices is distinguished. This distinguished vertex is called the root. We only consider unordered trees, that is,
trees in which the children of any vertex are unordered. Given a finite set $A$, we let $\m(A) = \min(A)$ and $\M(A) = \max(A)$.

Let $T$ be a labelled tree with vertex labels given by the finite set $A$ and rooted at the vertex labelled with $\m(A)$. We direct the edges of $T$ away from the root so that if
$(i,j)$ is an edge in $T$ then $i$ is on the unique path from the root of $T$ to $j$. In this case we call $i$ the father of $j$ in $T$ and denote this by $\cf_T(j)$. A vertex
is said to be increasing if $\cf_T(i) < i$ and is decreasing otherwise. Note that $\cf_T(\m(A))$ is not defined and so the root of $T$ is neither increasing nor decreasing. We
say that a tree $T$ is increasing if every non root vertex in $T$ is an increasing vertex.

Given a tree $T$ and a vertex $i \in V(T)$ we define the hook generated by $i$ in $T$, written $\hook_T(i)$, to be the set of vertices $j$ such that $i$
is in the unique path from the root to $j$ in $T$. In other words, $\hook_T(i)$ is the set of vertices in the subtree of $T$ rooted at $i$. Note
that $i \in \hook_T(i)$.

Consider the family $\fT_A$ of increasing unordered labelled trees with vertex labels given by $A$. F\'eray, Goulden and Lascoux\cite{FGL2} studied a combinatorial sum
involving a hook weight summed over increasing trees with a fixed number of vertices. Using a generalization of the Pr\"ufer code, it was shown that these
sums have an appealing multiplicative closed form. In particular, the following theorem is proven.

\begin{thm}[Theorem 1.1 in \cite{FGL2}]\label{Thm.2}
	For a tree $T \in \fT_A$ define a weight on $T$ as
		\[ \widetilde{w}(T) = \prod_{i \in A \backslash \m(A)} x_{\cf_T(i)} \left( \sum_{ j \in \hook_T(i) } y_{i,j} \right). \]
	Then the generating series is given by
		\begin{align*}
			\Ta_A	&= \sum_{T \in \fT_A} \widetilde{w}(T) \\
				&= x_{\m(A)} y_{\M(A),\M(A)} \prod_{i \in A \backslash \{\M(A),\m(A)\}} \left( y_{i,i}\sum_{\substack{j \in A \\ j \leq i}} x_j + x_i \sum_{\substack{j \in A \\ j > i}} y_{i,j} \right).
		\end{align*}
\end{thm}

If one makes the specialization $x_i \to 1$ and $y_{i,j} \to 1$ for all $i$ and $j$ in Theorem~\ref{Thm.2} then the right hand side of the identity 
becomes $|A|^{|A|-2}$, which is the number of Cayley trees with vertices labelled by the set $A$ as shown by Cayley\cite{Cayley}. In other words, $|A|^{|A|-2}$ is the number of trees
with labels given by $A$ and which are not rooted and not necessarily increasing (although for convenience we may assume that a Cayley tree is rooted at the vertex labelled by $\m(A)$).
This observation prompted F\'eray, Goulden and Lascoux to ask for a combinatorial bijection between increasing trees and Cayley trees which could be used to prove
Theorem~\ref{Thm.2}. Further evidence for the existence of such a bijection was provided by some results in F\'eray and Goulden's earlier paper\cite{FG1} in which the authors
study a specialization of Theorem~\ref{Thm.2} and are able to give a combinatorial bijection for the top degree of the polynomial identity (Section 2.2 in \cite{FG1}) which
involves Cayley trees. In Section~\ref{Section.Direct} we give a bijective proof of Theorem~\ref{Thm.2} which involves Cayley trees, solving the problem posed by
F\'eray, Goulden and Lascoux.

In addition to answering the question posed by F\'eray, Goulden and Lascoux, the contents of Section~\ref{Section.Direct} also indicates a connection between the hook sum
formula in Theorem~\ref{Thm.2} and the theory of tree volume formulas defined by
Kelmans\cite{K92} and further studied by Kelmans, Postnikov and Pitman\cite{P01,P02,KP08}. This connection comes from Theorem~\ref{Bijection} below which implies that the
generating polynomial $\Ta_A$ in Theorem~\ref{Thm.2} is in fact a tree volume polynomial corresponding to the complete graph. More generally, this gives a connection between
the hook sum formula in Theorem~\ref{Thm.2} and various generalizations of the binomial theorem, such as Abel and Hurwitz' identities.

As an application of Theorem~\ref{Thm.2} we will take a moment to discuss more directly the connection to a multivariate generalization of the binomial theorem.
In \cite{Strehl}, Strehl proves the following multivariate generalization of the binomial theorem.
\begin{thm}[Theorem 1(7) in \cite{Strehl}]\label{Thm.S}
	Suppose $A$ is a finite set of positive integers and let
		\[ w_A(z) = z \prod_{i \in A \backslash \{\M(A)\}} \left( z + \sum_{\substack{j \in A \\ j \leq i}} x_j + \sum_{\substack{j \in A \\ j > i}} y_{i,j} \right). \]
	Then
		\[ w_A(u+v) = \sum_{B \sqcup C = A} w_B(u)w_C(v), \]
	where $B \sqcup C = A$ means that $B \cup C = A$ and $B \cap C = \emptyset$.
\end{thm}
By specializing variables, Theorem~\ref{Thm.S} can be seen to be a generalization of the binomial theorem. In particular, if we let $y_{i,j} \to 0$ and $x_i \to 0$ for all $i$ and $j$
then it is easily seen that the identity in Theorem~\ref{Thm.S} is the binomial identity. If we let $y_{i,j} \to 1$ and $x_i \to 1$ for all $i$ and $j$ then Theorem~\ref{Thm.S} gives
Abel's generalization\cite{Abel,Riordan} of the binomial theorem,
	\[ (u+v)(u+v+n)^{n-1} = \sum_{k = 0}^n \binom{n}{k} u (u+k)^{k-1}v(v+(n-k))^{n-k-1}, \]
where $n = |A|$. Lastly, if we let $y_{i,j} \to x_j$ for $i < j$ then Theorem~\ref{Thm.S} becomes Hurwitz' generalization\cite{Hurwitz} of Abel's identity,
	\[ (u+v)\left(u+v+ \sum_{i \in A} x_i\right)^{|A|-1} = \sum_{B \sqcup C = A} u\left(u+\sum_{i \in B} x_i\right)^{|B|-1} v \left(v+\sum_{i \in C}x_i\right)^{|C|-1}. \]
We refer the reader to Strehl's paper \cite{Strehl} for additional specializations of interest as well as a number of applications.

Theorem~\ref{Thm.2} can be used to give a new proof of Theorem~\ref{Thm.S}.
\begin{thm}\label{Thm.Binom}
	Let $x_i, i \geq 0$ and $y_{i,j}, 0 \leq i < j$ be indeterminates and for any finite set $A$ of integers let
 		\[ \Ta_A = x_{\m(A)} y_{\M(A),\M(A)} \prod_{i \in A \backslash \{\M(A),\m(A)\}} \left( y_{i,i}\sum_{\substack{j \in A \\ j \leq i}} x_j + x_i \sum_{\substack{j \in A \\ j > i}} y_{i,j} \right). \]
 	Then for any finite set $A$ of positive integers,
 		\[ \left. \Ta_{A \cup \{0\}} \right|_{x_0 = u + v} = \sum_{B \sqcup C = A} \left. \Ta_{B \cup \{0\}} \right|_{x_0 = u} \left. \Ta_{C \cup \{0\}} \right|_{x_0 = v}. \]
\end{thm}
\begin{proof}
	Our Theorem~\ref{Thm.Binom} above follows directly from Theorem~\ref{Thm.2} since both sides
	of the equality in Theorem~\ref{Thm.Binom} count trees in which each root edge is coloured either red or blue and then each blue edge is marked with a $u$ and each
	red edge is marked with a $v$.
\end{proof}
Note that if $A$ is a finite set of positive integers and we let $y_{i,i} \to 1$ for all $i$ and $y_{i,j} \to \frac{y_{i,j}}{x_i}$ for all $i < j$ then we recover Theorem~\ref{Thm.S}
from Theorem~\ref{Thm.Binom} where $w_A(z) = \left. \Ta_{A \cup \{0\}} \right|_{x_0 = z}$. Similarly, if we let $z \to x_0 y_{0,0}$, $x_i \to x_i y_{i,i}$ for $i \in A$ and
$y_{i,j} \to x_i y_{i,j}$ for $i < j \in A$ then we recover Theorem~\ref{Thm.Binom} from Theorem~\ref{Thm.S} where $\Ta_{A \cup \{0\}} = w_A$.

It should be noted that the method of proof for Theorem~\ref{Thm.S} and Theorem~\ref{Thm.Binom} is very similar, the main difference being the combinatorial description of
the generating series involved. Strehl uses the description of the generating series given in Proposition~\ref{MatrixTreeProp} below as sums over Cayley trees. Instead, we use
the description of the generating series as hook weighted sums over increasing trees as given in the statement of Theorem~\ref{Thm.2}.

The remainder of this paper is organized as follows. In Section~\ref{Section.Direct} we give a combinatorial proof of Theorem~\ref{Thm.2} by describing an `unsorting' operation
which can be applied to increasing trees and, after repeated application, results in a Cayley tree. Following the combinatorial proof we also describe two simple
proofs of Theorem~\ref{Thm.2}. In Section~\ref{Section.Indirect} we give an indirect combinatorial proof by showing that both expressions for the
polynomials $\Ta_A$ given in Theorem~\ref{Thm.2} satisfy the same recursion and initial conditions and in Section~\ref{Section.Algebraic} we give a direct algebraic proof
which uses the fact that increasing trees can be constructed inductively by adding leaves.

\section{A Bijective Proof}\label{Section.Direct}

Let $\cL_{i,j}(A)$ be the set of pairs $(T, \phi)$ where $T$ is a labelled tree with vertex labels given by $A$, $\phi$ is a function from the set of increasing vertices in $T$ to $A$
and the pair $(T,\phi)$ satisfies the following conditions.
\begin{enumerate}
  \item [(1)]	For any increasing vertex $v$ in $T$, $\phi(v) \in \hook_T(v)$ and $\phi(v) \geq v$.
  \item [(2)]	If $v$ is an increasing vertex in $T$ and $\phi(v) \not = v$ then every vertex on the unique path from $v$ to the root (not including the root) is increasing.
  \item [(3)]	If $v$ is a decreasing vertex in $T$ and $u$ is an increasing vertex with $\phi(u) \not = u$ then $u < v$.
  \item [(4)]	$T$ has $i$ decreasing vertices.
  \item [(5)]	$T$ has $j$ increasing vertices $v$ with $\phi(v) \not = v$. 
\end{enumerate}
Define a weight function on $\cL_{i,j}(n)$ by
	\[ \om(T, \phi) = \prod_{\substack{increasing \\ w \in V(T)}} x_{\cf_T(w)} y_{w,\phi(w)} \prod_{\substack{decreasing \\ w \in V(T)}} x_w y_{w,\cf_T(w)}. \]

The goal of the following theorem is to describe a method by which we can transform trees contained in the sets $\cL_{0,j}(n)$ (increasing trees) into trees counted
by the sets $\cL_{i,0}(n)$. The reason for this is that the increasing trees contained in the $\cL_{0,j}(n)$ sets are the objects of interest for the purposes of Theorem~\ref{Thm.2},
however, the weight function depends on non-local information. In particular, for some vertex $v$ it may be the case that $\phi(v)$ is not adjacent to $v$ since the only condition
is that $\phi(v) \in \hook_T(v)$. Fortunately, the following theorem says that we can repeatedly `unsort' the increasing trees so that they become Cayley trees in which the weight
function is entirely local. That is, in $\cL_{i,0}(n)$ the weight of each tree depends only on vertices and their neighbors and so the generating series can be computed in
a straightforward way.

\begin{thm}\label{Bijection}
	There exists a weight preserving bijection between $\cL_{i,j}(A)$ and $\cL_{i+1,j-1}(A).$
\end{thm}
\begin{proof}
	Let $(T,\phi) \in \cL_{i,j}(A)$ and let $v$ be the increasing vertex with greatest label such that $\phi(v) \not = v$. Let $b \in \hook_v(T)$ be the vertex adjacent to $v$ with $\phi(v) \in \hook_T(b)$ and
	$a$ be the vertex adjacent to $v$ on the unique path from $v$ to the root. In other words, $a = \cf_T(v)$ and $b$ is the child of $v$ whose hook contains $\phi(v)$. Note that condition 2 on $(T,\phi)$ implies
	that every vertex in the path from $v$ to the root is increasing.
	
	Form a new tree $T'$ by removing edges $av$ and $vb$ and adding edges $ab$ and $\phi(v) v$. Since condition 3 implies that $b$ is an increasing vertex in $T$, it is still increasing in $T'$. Also, vertex
	$v$ is decreasing in $T'$ by condition 1. If we construct a function $\phi'$ from the set of increasing vertices in $T'$ to $\{1, 2, \cdots, n\}$ such that $\phi'(u) = \phi(u)$ for all increasing $u$ in $T'$
	then it is easily checked that $(T', \phi')$ satisfies the conditions for $\cL_{i+1,j-1}(A)$.
	
	To see that this map is invertible we only need for $a, b$ and $v$ to be uniquely determined in $T'$ since $\phi$ must be equal to $\phi'$ for all increasing vertices in $T'$ and $\phi(v) = \cf_{T'}(v)$.
	However, $v$ is the unique decreasing vertex in $T'$ with the smallest label (this follows from condition 4 and the choice of $v$ in $T$). Once we know this, $a$ and $b$ must be the unique pair of adjacent
	vertices on the path from $v$ to the root in $T'$ such that $a < v < b$ and every vertex on the path from the root to $a$ in $T'$ is increasing (this follows from conditions 2 and 3).
	
	It is then straightforward to check that the constructed bijection is weight preserving since the weight corresponding to vertices $v$ and $b$ in $(T,\phi)$ is equal to their weight
	in $(T', \phi')$ (although $v$ becomes decreasing in $T'$).
\end{proof}

Now we need to determine the generating series for trees in the collection of sets of the form $\cL_{i,0}(n)$. However, note that in this case the map $\phi$ is redundant since every
vertex $v$ in such a tree must have $\phi(v) = v$. In other words, this amounts to determining the generating series for the set of Cayley trees.
\begin{pro}\label{MatrixTreeProp}
	Let $\cD(A)$  be the set of Cayley trees labelled by $A$, rooted at $\m(A)$ and with edges directed toward the root. For $T \in \cD(A)$ let
		\[ w(T) = \prod_{(i,j) \in E(T)} w_{i,j}, \]
	where
		\[ w_{i,j} = 	\begin{cases}
							x_i y_{i,j}	&	\mbox{ if } i < j, \\
							x_j y_{i,i}	&	\mbox{ if } i > j.
						\end{cases} \]
	Then
		\[ \sum_{T \in \cD(A)} w(T) = x_{\m(A)} y_{\M(A),\M(A)} \prod_{i \in A \backslash \{\M(A),\m(A)\}} \left( y_{i,i} \sum_{\substack{j \in A \\ j \leq i}} x_j + x_i \sum_{\substack{j \in A \\ j > i}} y_{i,j} \right). \]
\end{pro} 
\begin{proof}
	This follows from a straightforward application of the matrix tree theorem and is essentially the same as the method used in part of the proof of Proposition~1 in Strehl\cite{Strehl}.
	Without loss of generality we may assume that $A = \{1, 2, \cdots, n\}$.
			\[ \sum_{T \in \cD(\{1, 2, \cdots, n\})} w(T) = \det(K_{1,1}), \]
	where
		\[ k_{i,j} = 	\begin{cases}
							-x_i y_{i,j}	&	\mbox{ if } 1 \leq i < j \leq n, \\
							-x_j y_{i,i}	&	\mbox{ if } 1 \leq j < i \leq n, \\
							y_{i,i} \sum_{m=1}^{i-1} x_m + x_i \sum_{m = i+1}^n y_{i,m}	&	\mbox{ if } i = j,
						\end{cases} \]
	and $K_{1,1}$ is the matrix $K$ with the first row and column removed. Adding each of the columns to the last column and then subtracting $\frac{y_{i-1,i-1}}{y_{i,i}}$ times row $i$ from row $i-1$ for
	each $i$ gives the matrix $L$ with
		\[ \ell_{i,n} = 0 \mbox{ for } 1 \leq i < n-1, \qquad \ell_{i,j} = 0 \mbox{ for } 1 \leq j < i \leq n-1, \]
		\[ \ell_{i,i} = y_{i+1,i+1} \sum_{j=1}^{i+1} x_j + x_{i+1} \sum_{j=i+2}^n y_{i+1,j} \mbox{ for } 1 \leq i < n-1, \]
	and $\ell_{n-1,n-1} = y_{n,n}x_1.$
	Since $\det(K_{1,1}) = \det(L)$ is the product of the main diagonal of $L$, the result follows.
\end{proof}

\begin{proof}[Proof of Theorem~\ref{Thm.2}.]
	First note that by expanding it is easily seen that
		\[ \sum_{T \in \fT_A} \prod_{v = 2}^n x_{\cf_T(v)} \left( \sum_{u \in \hook_T(v)} y_{v,u} \right) = \sum_{(T,\phi)} \prod_{v=2}^n x_{\cf_T(v)} y_{v,\phi(v)} \]
	where the sum is over all pairs $(T, \phi)$ where $T$ is in $\fT_A$ and $\phi$ is a map from the set of increasing vertices in $T$ to $A$ with $\phi(v) \in \hook_T(v)$
	for all increasing vertices $v$. However, this is equal to the sum
		\[ \sum_{i \geq 0} \sum_{(T,\phi) \in \cL_{0,i}(A)} \om(T,\phi). \]
	Applying Theorem~\ref{Bijection} then gives, letting $\cD(A)$ be the set of Cayley trees with vertex labels given by $A$ as in Proposition~\ref{MatrixTreeProp},
		\[ \sum_{j \geq 0} \sum_{(T,\phi) \in \cL_{j,0}(A)} \om(T,\phi) = \sum_{T \in \cD(A)} \prod_{\substack{increasing \\ v \in V(T)}} x_{\cf_T(v)}y_{v,v} \prod_{\substack{decreasing \\ v \in V(T)}} x_v y_{v,\cf_T(v)}. \]
	The result then follows by applying Proposition~\ref{MatrixTreeProp}
\end{proof}

\section{An Indirect Combinatorial Proof}\label{Section.Indirect}

We will now give an indirect combinatorial proof of Theorem~\ref{Thm.2} which relies on Proposition~\ref{MatrixTreeProp}. Given a set $A$ of positive integers, let $t_A = 1$ if $|A| = 1$ and for $|A| > 1$,
	\[ t_A = x_{\m(A)} y_{\M(A),\M(A)} \prod_{i \in A \backslash \{\M(A),\m(A)\}} \left( y_{i,i} \sum_{\substack{j \in A \\ j \leq i}} x_j + x_i \sum_{\substack{j \in A \\ j > i}} y_{i,j} \right). \]
From Proposition~\ref{MatrixTreeProp} above we immediately get the following two results.
\begin{lem}\label{singleEdge}
	Let $\cE(A)$ be the subset of trees in $\cD(A)$ which have a unique edge incident with $\m(A)$. Then if
		\[ r_A = \sum_{T \in \cE(A)} w(T), \]
	with $w(T)$ as defined in Proposition~\ref{MatrixTreeProp}, then
		\[ r_A = x_{\m(A)} y_{\M(A),\M(A)}\prod_{i \in A \backslash \{\M(A),\m(A)\}} \left( y_{i,i} \sum_{\substack{j \in A \backslash \m(A) \\ j \leq i}} x_j + x_i \sum_{\substack{j \in A \backslash \m(A) \\ j > i}} y_{i,j} \right). \]
\end{lem}
\begin{proof}
	This follows from the observation that
		\[ r_A = x_{\m(A)} \left. \frac{d}{d x_{\m(A)}} t_A \right|_{x_{\m(A)} = 0}. \]
\end{proof}
\begin{pro}\label{LittleRecursion}
	With the polynomials $t_A$ and $r_A$ as defined above with $|A| > 1$ and for any $a \in A \backslash \{ \m(A) \}$,
		\[ t_A = \sum_{\substack{B \sqcup C = A \\ \m(A) \in C \\ a \in B}} x_{\m(A)} \left( \sum_{j \in B} y_{\m(B), j} \right) t_B t_C. \]
\end{pro}
\begin{proof}
	By Proposition~\ref{MatrixTreeProp} we know that $t_A = \sum_{T \in \cD(A)} w(T)$. For any $T \in \cD(A)$ there is a unique child $v$ of $\m(A)$ for which the subtree
	rooted at $v$ contains $a$. Letting $B$ be the set of labels in the subtree it follows from Lemma~\ref{singleEdge} that $w(T) = w(T_1)w(T_2)$ where $T_1 \in \cE(B \cup \m(A))$
	and $T_2 \in \cD(A \backslash B)$. Thus,
		\[ t_A = \sum_{\substack{B \sqcup C \\ \m(A) \in C \\ a \in B}} r_{B \cup \m(A)} t_C. \]
	We also see that
		\[ r_{B \cup \m(A)} = x_{\m(A)} \left( \sum_{j \in B} y_{\m(B),j} \right) t_B, \]
	from which the result follows.
\end{proof}

Given a finite set $A$ of positive integers let
	\[ \Ga_A = x_{\m(A)} \left. \frac{d}{d x_{\m(A)}} \Ta_A \right|_{x_{\m(A)} = 0} = \sum_{T \in \fR_A} \widetilde{w}(T) \]
where $\fR_A$ is the subset of $\fT_A$ in which there is a single edge incident with $\m(A)$.
\begin{pro}\label{BigRecursion}
	With $|A| > 1$ and for any $a \in A \backslash \{\m(A)\}$,
		\[ \Ta_A = \sum_{\substack{B \sqcup C = A \\ \m(A) \in C \\ a \in B}} x_{\m(A)} \left( \sum_{j \in B} y_{\m(B), j} \right) \Ta_B \Ta_C. \]
\end{pro}
\begin{proof}
	As in the proof of Proposition~\ref{LittleRecursion} by considering the subtree of $\m(A)$ which contains $a$ we see that
		\[ \Ta_A = \sum_{\substack{B \sqcup C = A \\ \m(A) \in C \\ a \in B}} \Ga_{B \cup \m(A)} \Ta_C. \]
	Since for any tree $T \in \fR(B \cup \m(A))$ we have $\hook_T(\m(B)) = B$ this gives
		\[ \Ga_{B \cup \m(A)} = x_{\m(A)} \left( \sum_{j \in B} y_{\m(B),j} \right) \Ta_B \]
	from which the result follows.
\end{proof}

\begin{proof}[Proof of Theorem~\ref{Thm.2}.]
That
	\[ \Ta_A = x_{\m(A)} y_{\M(A),\M(A)} \prod_{i \in A \backslash \{\M(A),\m(A)\}} \left( y_{i,i}\sum_{\substack{j \in A \\ j \leq i}} x_j + x_i \sum_{\substack{j \in A \\ j > i}} y_{i,j} \right), \]
for $|A| > 1$ follows by induction after comparing Proposition~\ref{LittleRecursion} and Proposition~\ref{BigRecursion} and checking the base case
	\[ \Ta_A = 1 = t_A, \]
when $|A| = 1$.
\end{proof}

\begin{rem}
	The indirect combinatorial proof above uses the canonical decomposition of an unordered increasing tree by removing the edge with vertex labels $1$ and $2$. The same result can be obtained by
	using the decomposition in which the vertex labelled $1$ is removed. In either case the proof is essentially the same, the generating series for hook-weighted increasing trees and weighted
	labelled trees are shown to satisfy the same recursion.
\end{rem}

\section{An Algebraic Proof}\label{Section.Algebraic}

Lastly we give an algebraic proof of Theorem~\ref{Thm.2} which proceeds by induction on the number of vertices. In fact, we prove a small variation of Theorem~\ref{Thm.2} as it
will make the algebraic manipulations that follow a little easier.
\begin{thm}[Variation on Theorem~\ref{Thm.2}]\label{Thm.3}
Let $\fT_n = \fT_{\{1, 2, \cdots, n\}}$ and let
	\[ \Ta_n = \sum_{T \in \fT_n} \left( \prod_{i = 2}^n x_{\cf_T(i)} \right)\left( \prod_{i = 1}^n \left( \sum_{j \in \hook_T(i)} y_{i,j} \right) \right). \]
Then for $n \geq 1$,
	\[ \Ta_n = y_{n,n} \prod_{i=1}^{n-1} \left( y_{i,i} \sum_{j=1}^i x_j + x_i \sum_{j=i+1}^n y_{i,j} \right). \]
\end{thm}

Note that Theorem~\ref{Thm.2} follows very easily from Theorem~\ref{Thm.3}.
\begin{proof}[Proof of Theorem~\ref{Thm.2}.]

Without loss of generality we may assume that the set $A$ in Theorem~\ref{Thm.2} is $A = \{1, 2, \cdots, n\}$. In this case,
	\[ \Ta_n = \left( \sum_{i = 1}^n y_{1,i} \right) \Ta_A, \]
and so the result follows.
\end{proof}

\begin{proof}[Proof of Theorem~\ref{Thm.3}.]

First, notice that
	\[ \Ta_1 = y_{1,1}, \qquad \mbox{ and } \qquad \Ta_2 = x_1 (y_{1,1} + y_{1,2}) y_{2,2} \]
agree with the combinatorial definition. Now, suppose that $\Ta_n$ is as above and let, for $1 \leq i < j \leq n$, $\psi^i_j$ be the evaluation map which takes
$y_{k,i}$ to $y_{k,i} + y_{k,j}$ for all $1 \leq k \leq i$. Then since every increasing tree on $n+1$ vertices is created by adding the vertex labelled $n+1$ to
some other vertex, we see that
	\[ \Ta_{n+1} = y_{n+1,n+1} \sum_{i=1}^n x_i \psi^i_{n+1} \Ta_n. \]
Let $\al_i(n) = y_{i,i} \sum_{j=1}^i x_j + x_i(\sum_{j=i+1}^n y_{i,j})$ so that $\Ta_n = y_{n,n} \prod_{i=1}^{n-1} \al_i(n)$.
For $1 \leq i \leq n-1$ we have
\begin{align*}
	\psi^i_{n+1} \Ta_n	&= y_{n,n} \prod_{k=1}^{n-1} \left( \psi^i_{n+1} y_{k,k} \sum_{j=1}^k x_j + x_k \sum_{j=k+1}^n \psi^i_{n+1} y_{k,j} \right) \\
						&= y_{n,n} \left( \prod_{k=1}^i \al_k(n+1) \prod_{k=i+1}^{n-1} \al_k(n) + \sum_{j=1}^{i-1} x_j y_{i,n+1} \prod_{k=1}^{i-1} \al_k(n+1) \prod_{k=i+1}^{n-1} \al_k(n) \right).
\end{align*}
If we let
	\[ \be_i(n) = \prod_{k=1}^{i-1} \al_k(n+1) \prod_{k=i+1}^{n-1} \al_k(n), \]
this shows that for $1 \leq i \leq n-1$,
	\[ \psi^i_{n+1} \Ta_n = y_{n,n} \left( \prod_{k=1}^i \al_k(n+1) \prod_{k=i+1}^{n-1} \al_k(n) + \sum_{j=1}^{i-1} x_j y_{i,n+1} \be_i(n) \right). \]
Also,
	\[ \psi^n_{n+1} \Ta_n = (y_{n,n} + y_{n,n+1}) \prod_{k=1}^{n-1} \al_k(n+1). \]
Putting this together gives, after some algebraic manipulation,
\begin{align*}
	\frac{\Ta_{n+1}}{y_{n+1,n+1}}	&= \sum_{i=1}^n \psi^i_{n+1} \Ta_n \\
									&= x_n(y_{n,n} + y_{n,n+1}) \prod_{k=1}^{n-1} \al_k(n+1) \\
									&\qquad + \sum_{i=1}^{n-1} y_{n,n} x_i \left( \prod_{k=1}^i \al_k(n+1) \prod_{k=i+1}^{n-1} \al_k(n) + \sum_{j=i+1}^{n-1} x_j y_{j,n+1} \be_j(n) \right).
\end{align*}
Now, since $\al_k(n+1) = \al_k(n) + x_k y_{k,n+1}$, by expanding from the largest index to the smallest,
\[ \prod_{k=1}^{n-1} \al_k(n+1) = \prod_{k=1}^i \al_k(n+1) \prod_{k=i+1}^{n-1} \al_k(n) + \sum_{j=i+1}^{n-1} x_j y_{j,n+1} \be_j(n). \]
Thus,
\begin{align*}
	\frac{\Ta_{n+1}}{y_{n+1,n+1}}	&= x_n(y_{n,n} + y_{n,n+1}) \prod_{k=1}^{n-1} \al_k(n+1) + \sum_{i=1}^{n-1} y_{n,n} x_i \prod_{k=1}^{n-1} \al_k(n+1) \\
									&= \prod_{k=1}^{n-1} \al_k(n+1) \left( y_{n,n} \sum_{i=1}^n x_i + x_n y_{n,n+1} \right) \\
									&= \prod_{k=1}^n \al_k(n+1).
\end{align*}
\end{proof}

\bibliographystyle{plain}
\bibliography{document}

\begin{thebibliography}{10}

\bibitem{Abel}
N.H. Abel.
\newblock {B}eweis eines {A}usdruckes, von welchem die {B}inomial-{F}ormel ein
  einzelner {F}all ist.
\newblock {\em Crelle's J. Reine Angew. Math.}, 1:159--160, 1826.

\bibitem{Cayley}
A.~Cayley.
\newblock A theorem on trees.
\newblock {\em Quart. J. Math}, 23:376--378, 1889.

\bibitem{FGL2}
V.~{F{\'e}ray}, I.~P. {Goulden}, and A.~{Lascoux}.
\newblock {An edge-weighted hook formula for labelled trees}.
\newblock {\em To appear in J. of Combin., ar{X}iv 1310.4093}.

\bibitem{FG1}
Valentin F{\'e}ray and I.~P. Goulden.
\newblock A multivariate hook formula for labelled trees.
\newblock {\em J. Combin. Theory Ser. A}, 120(4):944--959, 2013.

\bibitem{Hurwitz}
A.~Hurwitz.
\newblock {\"U}ber {A}bel's {V}erallgemeinerung der binomischen {F}ormel.
\newblock {\em Acta Math.}, 26:199--203, 1902.

\bibitem{K92}
A.~K. Kelmans.
\newblock Spanning trees of extended graphs.
\newblock {\em Combinatorica}, 12(1):45--51, 1992.

\bibitem{KP08}
Alexander Kelmans and Alexander Postnikov.
\newblock Generalizations of {A}bel's and {H}urwitz's identities.
\newblock {\em European J. Combin.}, 29(7):1535--1543, 2008.

\bibitem{P01}
Jim Pitman.
\newblock Random mappings, forests, and subsets associated with
  {A}bel-{C}ayley-{H}urwitz multinomial expansions.
\newblock {\em S\'em. Lothar. Combin.}, 46:Art. B46h, 45 pp. (electronic),
  2001/02.

\bibitem{P02}
Jim Pitman.
\newblock Forest volume decompositions and {A}bel-{C}ayley-{H}urwitz
  multinomial expansions.
\newblock {\em J. Combin. Theory Ser. A}, 98(1):175--191, 2002.

\bibitem{Riordan}
John Riordan.
\newblock {\em Combinatorial identities}.
\newblock Robert E. Krieger Publishing Co., Huntington, N.Y., 1979.
\newblock Reprint of the 1968 original.

\bibitem{Strehl}
Volker Strehl.
\newblock Identities of {R}othe-{A}bel-{S}chl\"afli-{H}urwitz-type.
\newblock {\em Discrete Math.}, 99(1-3):321--340, 1992.

\end{thebibliography}

\end{document}